\documentclass[12pt,twoside]{article}
\usepackage{a4wide}
\usepackage{amssymb,bm}
\usepackage{amsmath}

\newtheorem{theorem}{Theorem}

\newtheorem{lemma}{Lemma}

\newenvironment{proof}{\noindent\textsc{{Proof}}.}{\hfill\raisebox{-1ex}{$\boxtimes$}}

\renewcommand{\Bbb}[1]{\mathbb{#1}}
\newcommand{\N}{{\Bbb N}}         
\newcommand{\R}{{\Bbb R}}         
\newcommand{\I}{{\Bbb I}}         %
\newcommand{\Rp}{{\Bbb R}^{+}}    
\newcommand{\Z}{{\Bbb Z}}         

\newcommand{\cC}{{\cal C}}

\newcommand{\cF}{{\cal F}}

\newcommand{\cH}{{\cal H}}

\newcommand{\cM}{{\cal M}}

\newcommand{\cS}{{\cal S}}
\newcommand{\cSM}{{\cal S}^\times}

\newcommand{\ve}{\varepsilon}

\def\K{{\rm {\bf K}_c^{\bm\theta}}}
\def\hK{{\rm {\bf K}_c}}

\newcommand{\diam}{\operatorname{diam}}

\newcommand{\vv}[1]{{\mathbf{#1}}}

\renewcommand{\le}{\leqslant}
\renewcommand{\ge}{\geqslant}

\newcommand{\bad}{\mathbf{Bad}}

\begin{document}

\title{A note on three problems in metric \\ Diophantine approximation}

\author{Victor Beresnevich\footnote{Research supported by EPSRC grant EP/J018260/1}
\\ {\small\sc (York) } \and  Sanju Velani\footnote{Research supported by EPSRC grants EP/E061613/1, EP/F027028/1 and EP/J018260/1 } \\ {\small\sc (York)}}

\date{\small\today}

\maketitle

\centerline{{\it Dedicated to  Shrikrishna Gopalrao Dani}}
\centerline{{\it on the occasion of his 65th birthday}}

\begin{abstract}
{\footnotesize
The use of Hausdorff measures and dimension in the theory of Diophantine approximation dates back to the 1920s with the theorems of Jarn\'ik and Besicovitch regarding well-approximable and badly-approximable points. In this paper we consider three inhomogeneous problems that further develop these classical results.
Firstly, we obtain a  Jarn\'ik type theorem for the set $ \cS^\times_2(\psi;\bm\theta)$ of multiplicatively approximable points in the plane  $\R^2$.  This Hausdorff measure statement does not reduce to Gallagher's Lebesgue measure statement as one might expect and is new even in the homogeneous setting ($\bm\theta = \bm0$). Next, we establish a  Jarn\'ik type theorem for the set $ \cS^\times_2(\psi;\bm\theta)\cap \cC $ where $\cC$ is a non-degenerate planar curve.  This completes the Hausdorff theory for planar curves  and clarifies a potential oversight in \cite{BL}.  Finally, we show that the set $ \bad(i,j;\bm\theta) $ of simultaneously inhomogeneously badly approximable points in $\R^2$ is of full dimension.  The underlying philosophy behind the proof has other applications; e.g. towards establishing the inhomogeneous version of Schmidt's Conjecture. The higher dimensional analogues of the planar results are also discussed.}
\end{abstract}

\vspace{2cm}

\bigskip


\section{Multiplicatively $\psi$-well approximable points}

Throughout $\psi:\N\to[0,+\infty)$ is a non-negative function. We will normally assume that $\psi$ is strictly positive and monotonically decreasing in which case $\psi$ will be referred to as an {\em approximating function}. Given $ \psi$, a real number $x$ will be called {\em $\psi$-well approximable} or simply {\em $\psi$-approximable} if there are infinitely many $q\in\N$ such that
\begin{equation*}
 \|qx\| \ <  \ \psi(q) \, .
\end{equation*}
Here and throughout $ \| \cdot  \| $ denotes the distance of a real
number to the nearest integer. Let $\cS_1(\psi)$ denote the set of all $\psi$-approximable real numbers.
The set $\cS_1(\psi)$ is invariant under translations by integers. Hence,  we will often restrict $x$ to lie in the unit interval $\I:=[0,1]$.

The well known theorem of Dirichlet states that $\cS_1(\psi)=\R$ when $\psi(q)=q^{-1}$. In turn, a rather simple consequence of the Borel-Cantelli lemma from probability theory is that $\cS_1(\psi)$ is null (that is of Lebesgue measure zero) whenever $\sum_{q=1}^\infty\psi(q)<\infty$. However, Khintchine's theorem \cite{Khintchine-1924} tells us that the set $\cS_1(\psi)$ is full (that is its complement is of Lebesgue measure zero) whenever $\sum_{q=1}^\infty\psi(q)=\infty$ and $\psi$ is monotonic. In order to quantify the size of $\cS_1(\psi)$ when it is null, Jarn\'ik \cite{Jarpa} and Besicovitch \cite{Besic34} pioneered the use of Hausdorff measures and dimension. Throughout, $\dim X$ will denote the Hausdorff dimension of a  subset $X$ of  $\R^n$ and $\cH^s(X)$ the $s$-dimensional Hausdorff measure (see \S\ref{HMD} for the definition and further details).   The modern version of the classical Jarn\'ik-Besicovitch theorem (see \cite{BBDVRoth} or \cite{BDV06}) states that for any approximating function $\psi$
\begin{equation}\label{JB}
\dim \cS_1(\psi) = \min\left\{1,\frac{2}{\tau+1}\right\}   \qquad\text{where }  \quad \tau:=\liminf_{q\to\infty}\frac{-\log\psi(q)}{\log q}\,.
\end{equation}
In other words, the `modern theorem'  relates the Hausdorff dimension of  $\cS_1(\psi)$ to  the lower order  at infinity of $1/\psi$ and up to a certain degree allows us to  discriminate between $\psi$-approximable sets of Lebesgue measure zero.
A more delicate measurement of the `size' of  $\cS_1(\psi)$ is obtained by expressing the size in terms of Hausdorff measures $\cH^s$.   With respect to such measures, the modern version of Jarn\'ik theorem  (see \cite{BBDVRoth} or \cite{BDV06})  states that  for any $s \in (0,1)$ and any  approximating function  $\psi$
\begin{equation}\label{jarnik}
\cH^{s}\big(\cS_1(\psi)\cap\I\big) \, = \,
\left\{\begin{array}{cl}
0  & {\rm if} \;\;\;
\textstyle{\sum_{q=1}^\infty} \; q^{1-s}\psi^{s}(q) \; <\infty \; ,\\[2ex]
\cH^s(\I) & {\rm if} \;\;\;
\textstyle{ \sum_{q=1}^\infty } \;  q^{1-s}\psi^{s}(q) \;
 =\infty \; .
\end{array}\right.
\end{equation}
Note that for $0<s<1$ we have that $\cH^s(\I)=\infty$. However,  since $\cH^1(\I)=1$,  the statement as written also holds for $s=1$ due to the aforementioned theorem of Khintchine.   Note that it is trivially true for $s>1$. The upshot is that statement \eqref{jarnik} is true for any $s >0 $  and is referred to as the {\em Khintchine-Jarn\'ik theorem}. It is worth pointing out that there is an even more general version of \eqref{jarnik} that makes use of more general Hausdorff measures, see \cite{BBDVRoth, BDV06, BV06}. Within this paper we restrict ourselves to the case of $s$-dimensional Hausdorff measures.

In higher dimensions there are various natural generalizations of $\cS_1(\psi)$. Given an approximating function $\psi$, the point $\vv x =(x_1,\dots,x_n) \in \R^n$ will be called  {\em $\psi$-well approximable} or simply {\em $\psi$-approximable} if there are infinitely many $q\in\N$ such that
\begin{equation}\label{vb1}
\max\{\|qx_1\|,\dots,\|qx_n\|\}<\psi(q)
\end{equation}
and it will be called {\em multiplicatively $\psi$-well approximable} or simply {\em multiplicatively $\psi$-approximable} if there are infinitely many $q\in\N$ such that
\begin{equation}\label{vb2}
\|qx_1\|\cdots\|qx_n\|<\psi(q)\,.
\end{equation}
Denote by  $\cS_n(\psi)$ the set of $\psi$-approximable points in $\R^n$ and by $\cSM_n(\psi)$ the set of multiplicatively $\psi$-approximable points in $\R^n$. On comparing \eqref{vb1} and \eqref{vb2} one easily spots that $$\cS_n(\psi^{1/n})\subset\cSM_n(\psi)  \, . $$

For the sake of clarity, in what follows we will mainly restrict our attention to the case of the plane $\R^2$. The Khintchine-Jarn\'ik theorem for $\cS_2(\psi)$ (see \cite{BBDVRoth} or \cite{BDV06}) states that for any $s>0$ and any approximating function $\psi$
\begin{equation}\label{KJ}
\cH^{s}\big(\cS_2(\psi)\cap\I^2\big) \, = \,
\left\{\begin{array}{cl}
0  & {\rm if} \;\;\;
\textstyle{\sum_{q=1}^\infty} \; q^{2-s}\psi^{s}(q) \; <\infty \; ,\\
&
\\
\cH^s(\I^2) & {\rm if} \;\;\;
\textstyle{ \sum_{q=1}^\infty } \;  q^{2-s}\psi^{s}(q) \;
 =\infty \; .
\end{array}\right.
\end{equation}
Regarding the Lebesgue case,  which corresponds to when $s=2$, Gallagher \cite{gall} showed that the monotonicity of $\psi$ is unnecessary. As a consequence of the Mass Transference Principle \cite{BV06} we have that \eqref{KJ} holds for any $\psi$ (not necessarily monotonic) and any $s>0$.

In the multiplicative setup,  Gallagher \cite{gallagher1962} essentially proved that for any approximating function~$\psi$
\begin{equation}\label{Gal}
\cH^2\big(\cSM_2(\psi)\cap\I^2\big) \, = \,
\left\{\begin{array}{cl}
0  & {\rm if} \;\;\;
\textstyle{\sum_{q=1}^\infty} \; \psi(q)\log q \; <\infty \; ,\\
&
\\
\cH^2(\I^2) & {\rm if} \;\;\;
\textstyle{ \sum_{q=1}^\infty } \;  \psi(q)\log q \;
 =\infty \; .
\end{array}\right.
\end{equation}
The extra log factor in the above sum accounts for the larger volume of the fundamental domains defined by \eqref{vb2} compared to \eqref{vb1}. The recent work \cite{BHV} has made an attempt to relax the monotonicity assumption on $\psi$ within the multiplicative setting. Our goal in this paper  is to investigate the Hausdorff measure theory within the multiplicative setting.

\medskip

\textbf{Problem 1:} \textit{Determine the Hausdorff measure $\cH^s $ of $\cSM_n(\psi)$.}

\medskip

\noindent This problem is somewhat  different to the non-multiplicative setting where we have the uniform solution  given by \eqref{KJ}. First of all,  we note that
\begin{equation}\label{vb4}
\text{if $ \ s\le 1  \ $ then $  \ \cH^s(\cSM_2(\psi)\cap\I^2)=\infty  \ $ irrespective of approximating function $ \ \psi $.}
\end{equation}
To see this,  we observe that for any $\psi$-approximable number $\alpha\in\R$ the whole line $x_1=\alpha$ is contained in $\cSM_2(\psi)$. Hence,
\begin{equation}\label{vb3}
\cS_1(\psi)\times\R\subset \cSM_2(\psi)\,.
\end{equation}
It is easy to verify   (for example, by using the theory of continued fractions) that $\cS_1(\psi)$ is an infinite set for any approximating function $\psi$ and so \eqref{vb3} implies \eqref{vb4}. Next, since  $\cSM_2(\psi)  \subseteq \R^2$, we trivially have that
 \begin{equation*}
\text{if  $ \ s > 2   \ $ then $  \ \cH^s(\cSM_2(\psi)\cap\I^2)= 0   \ $ irrespective of $ \ \psi $.}
\end{equation*}
 The upshot of this and \eqref{vb4} is that when attacking  Problem 1, there is no loss of generality in assuming that $ s \in (1,2]$.   Furthermore,  the Lebesgue case ($s=2$) is covered by Gallagher's result so we may as well assume that $1 < s < 2$.

Recall that Gallagher's multiplicative statement  \eqref{Gal} has the extra  `log factor' in the `volume' sum compared to the simultaneous statement  \eqref{KJ}.  It is therefore natural to expect the log factor to appear in one form or another when determining the  genuine  `fractal' Hausdorff measure $\cH^s$ of $\cSM_2(\psi)$; that is to say with $ s \in (1, 2)$.  This, as we shall soon see,  is very far from the truth. The `log factor' completely disappears!  Thus,  genuine `fractal'  Hausdorff measures  are  insensitive to the multiplicative nature of $\cSM_2(\psi)$. Indeed, what  we essentially have is  that
$$
\cH^{s}\big(\cS^\times_2(\psi)\big) \, =
\cH^{s-1}\big(\cS_1(\psi)\big) \,
$$
which reflects  the fact that for  genuine  `fractal'  Hausdorff measures the inclusion \eqref{vb3} is sharp. In short, for any $s \in (0,1)$, the points of $\cSM_2(\psi)$ that do not lie in $\R\times\cS_1(\psi)$ do not contribute any substantial `mass' in terms of the associated  $s$-dimensional Hausdorff measure.

The ideas and tricks used in our investigation of Problem 1 are equally valid within  the more general inhomogeneous setup: given an approximating function $\psi$ and a fixed point $\bm\theta=(\theta_1,\dots,\theta_n)\in\R^n$, let
$\cSM_n(\psi;\bm\theta)$ denote  the set of points $(x_1,\dots,x_n) \in \R^n$ such that there are infinitely many $q\in\N$ satisfying the inequality
\begin{equation}\label{vb2++}
\|qx_1-\theta_1\|\cdots\|qx_n-\theta_n\|<\psi(q)\,.
\end{equation}
We prove the following inhomogeneous statement.

\begin{theorem}\label{t1}
Let $\psi$ be an approximating function, $\bm\theta=(\theta_1,\theta_2)\in\R^2$ and $s \in (1,2)$. Then
\begin{equation}\label{e:006+}
\cH^{s}\big(\cS^\times_2(\psi;\bm\theta) \cap\I^2  \big) \, = \,
\left\{\begin{array}{ll}
0  & {\rm if} \;\;\;
\textstyle{\sum_{q=1}^\infty} \; q^{2-s}\psi^{s-1}(q) \; <\infty \; ,\\[1ex]
&
\\
\cH^{s}(\I^2) & {\rm if} \;\;\;
\textstyle{ \sum_{q=1}^\infty } \;  q^{2-s}\psi^{s-1}(q) \;
 =\infty \; .
\end{array}\right.
\end{equation}
\end{theorem}

\medskip

\noindent {\em Remark 1.1.}  Note that $\cH^{s}(\I^2)=\infty$ when $s<2$. We  reiterate the fact that unlike the Khintchine-Jarn\'ik theorem,  the statement of Theorem  \ref{t1} is false when $s=2$.

\medskip

\noindent {\em Remark 1.2.}  In higher dimensions, Gallagher's multiplicative statement reads
\begin{equation*}\label{Gal}
\cH^n\big(\cSM_n(\psi)\cap\I^2\big) \, = \,
\left\{\begin{array}{cl}
0  & {\rm if} \;\;\;
\textstyle{\sum_{q=1}^\infty} \; \psi(q)\log^{n-1}q \; <\infty \; ,\\
&
\\
\cH^2(\I^2) & {\rm if} \;\;\;
\textstyle{ \sum_{q=1}^\infty } \;  \psi(q)\log^{n-1}q \;
 =\infty \; .
\end{array}\right.
\end{equation*}
For $n > 2$, the  proof of Theorem~\ref{t1} can be adapted to show that for any $s\in(n-1,n)$
$$
\cH^{s}\big(\cS^\times_n(\psi;\bm\theta) \cap\I^n  \big) \, = \,
0   \qquad\text{if} \qquad
{\sum_{q=1}^\infty} \; q^{n-s}\psi^{s+1-n}(q)\log^{n-2}q \; <\infty \; .
$$
Thus, for convergence in higher dimensions we loose a log factor from the Lebesgue volume sum appearing in Gallagher's result.  This of course is absolutely consistent with the $n=2$ situation given by Theorem~\ref{t1}.
Regarding a  divergent statement,  the arguments used in proving  Theorem~\ref{t1} can be adapted to show that for any $s\in(n-1,n)$
$$
\cH^{s}\big(\cS^\times_n(\psi;\bm\theta) \cap\I^n  \big) \, = \,
\cH^{s}(\I^n)   \qquad\text{if} \qquad
{\sum_{q=1}^\infty} \; q^{n-s}\psi^{s+1-n}(q) \; =\infty \; .
$$
Thus,  there is a discrepancy in the above `$s$-volume' sum  conditions for convergence and divergence when $n > 2$. In view of this,  it remains an interesting open problem to determine the necessary and sufficient condition for $\cH^{s}\big(\cS^\times_n(\psi;\bm\theta) \cap\I^n  \big)$ to be zero or infinite in higher dimensions.

\subsection{Proof of Theorem~\ref{t1}}

To simplify notation the  symbols $\ll$
and $\gg$ will be used to indicate an inequality with an
unspecified positive multiplicative constant. If $a \ll b$ and $a
\gg b$ we write $ a \asymp b $, and  say that the quantities $a$
and $b$ are comparable.  For a real number $x$, the quantity $\{x\}$ will denote the fractional part of $x$  and $[x]$ the integer part of $x$.

Without loss of generality, throughout the proof of  Theorem \ref{t1}  we can assume that $\bm\theta=(\theta_1,\theta_2)\in\I^2$.

\subsubsection{A covering of $\cSM_2(\psi,\bm\theta) \cap \I^2$}

In this section we obtain an effective covering of the set $\cSM_2(\psi, \bm\theta)   \cap \I^2 $ that will be used in establishing the convergence case of Theorem~\ref{t1}.

\begin{lemma}\label{l1}
Let $0<\ve<1$, $(x_1,x_2)\in\I^2$, $(\theta_1,\theta_2)\in\I^2$, $q\in\N$ and
\begin{equation}\label{x1}
    \prod_{i=1}^2\|qx_i-\theta_i\|<\ve\,.
\end{equation}
Then there exist $m\in\Z$ and $p_1,p_2\in\{-1,0,\dots,q\}$ such that
$$
\|qx_i-\theta_i\|=|qx_i-\theta_i-p_i|\qquad\text{for }i=1,2\,,
$$
\begin{equation}\label{e1}
    \|qx_1-\theta_1\|<2^{m}\sqrt{2\ve},\qquad
    \|qx_2-\theta_2\|<2^{-m}\sqrt{2\ve}
\end{equation}
and
\begin{equation}\label{e1x}
2^{|m|}\sqrt{\ve}\le 1\,.
\end{equation}
\end{lemma}

\begin{proof}
The existence of $p_i\in\{-1,0,\dots,q\}$ with $|qx_i-p_i-\theta_i|=\|qx_i-\theta_i\|$ is an immediate consequence of the fact that $x_i,\theta_i\in\I$. Thus, the only thing that we need to prove is the existence of $m$ satisfying (\ref{e1}) and (\ref{e1x}).

If $\|qx_i-\theta_i\|<\sqrt{2\ve}$ for each $i=1,2$ then we can define $m=0$. In this case (\ref{e1}) is obvious and (\ref{e1x}) is a consequence of the fact that $0<\ve<1$.

Without loss of generality, assume that $\|qx_1-\theta_1\|\ge \sqrt{2\ve}$ and let $m\in\Z$ be the unique integer such that
$$
2^{m-1}\sqrt{2\ve}\le \|qx_1-\theta_1\|<2^{m}\sqrt{2\ve}\,.
$$
Since $\|qx_1-\theta_1\|\ge\sqrt{2\ve}$, we have that $m\ge0$. Furthermore, since $\|qx_1-\theta_1\|\le1/2$, we have that $2^m\sqrt\ve<1$ whence (\ref{e1x}) follows. The left hand side of (\ref{e1}) holds by the definition of $m$. To show the right hand side of (\ref{e1x}) we use (\ref{x1}). Indeed, we have that
$$
2^{m-1}\sqrt{2\ve}\|qx_2-\theta_2\|\le \prod_{i=1}^2\|qx_i-\theta_i\|<\ve
$$
whence the right hand side of (\ref{e1x})  follows. This completes the proof of the lemma.
\end{proof}

\begin{lemma}\label{l2}
Let $\psi:\N\to[0,1)$ be decreasing and $\bm\theta=(\theta_1,\theta_2)\in\I^2$. Then for any $\ell\in\N$
\begin{equation}
\cSM_2(\psi;\bm\theta) \cap \I^2  \ \subset  \ \bigcup_{t=\ell}^\infty\ \
\bigcup_{2^t\le q< 2^{t+1}}\ \bigcup_{\substack{m\in\Z\\[0.5ex] 2^{|m|}\sqrt{\psi(2^t)}\le 1}}\ \bigcup_{p_1=-1}^q \ \bigcup_{p_2=-1}^q S_{\bm\theta}(q,m,p_1,p_2) \ ,
\end{equation}
where
$$
S_{\bm\theta}(q,m,p_1,p_2):=\left\{(x_1,x_2)\in\I^2: \begin{array}{l}
\left|x_1-\dfrac{p_1+\theta_1}{q}\right|< \dfrac{2^{m}\sqrt{2\psi(2^t)}}{2^t}\\[3ex]
\left|x_2-\dfrac{p_2+\theta_2}{q}\right|< \dfrac{2^{-m}\sqrt{2\psi(2^t)}}{2^t}\
                                           \end{array}
 \right\}\,.
$$
\end{lemma}

\begin{proof}
It is easily verified that
$$
  \cSM_2(\psi;\bm\theta) \cap \I^2 \ =   \ \bigcap_{\ell=1}^\infty \ \bigcup_{t=\ell}^\infty\ \bigcup_{2^t\le q< 2^{t+1}}
\Big\{(x_1,x_2)\in\I^2 :\prod_{i=1}^2\|qx_i-\theta_i\|<\psi(q)\Big\}.
$$
Since $\psi$ is decreasing, $\psi(q)\le\psi(2^t)$ for  $2^t\le q<2^{t+1}$. Then,
for any $\ell\in\N$
\begin{equation}\label{e8}
\cSM_2(\psi;\bm\theta)  \ \subset \  \bigcup_{t=\ell}^\infty\
\bigcup_{2^t\le q< 2^{t+1}}\
\Big\{(x_1,x_2)\in\I^2 :\prod_{i=1}^2\|qx_i-\theta_i\|<\psi(2^t)\Big\}.
\end{equation}
For a fixed pair of $t$ and $q$ with $2^t\le q<2^{t+1}$, by Lemma~\ref{l1} with $\ve=\psi(2^t)$, we get that
$$
\Big\{(x_1,x_2)\in\I^2 :\prod_{i=1}^2\|qx_i-\theta_i\|<\psi(2^t)\Big\}\subset
\bigcup_{\substack{m\in\Z\\[0.5ex] 2^{|m|}\sqrt{\psi(2^t)}\le 1}}\ \bigcup_{p_1=-1}^q \ \bigcup_{p_2=-1}^q S_{\bm\theta}(q,m,p_1,p_2)\,.
$$
This together with (\ref{e8}) completes the proof of the lemma.
\end{proof}

\subsubsection{Hausdorff measure and dimension  \label{HMD}}

We briefly recall various facts regarding  Hausdorff measures that will be used in the course of establishing  Theorem~\ref{t1}.  Given $\delta>0$ and a set $X\subset\R^n$, any finite or countable collection $\{B_i\}$ of subsets of $\R^n$ such that
$$
X\subset \bigcup_i B_i\qquad(\text{i.e. }\{B_i\}\text{ is a cover for }X)
$$
and
$$
\diam B_i\le \delta\quad\text{for all }i
$$

\noindent is called a \emph{$\delta$-cover} of $X$. Given  a real number $s$, let
$$
\cH^s_\delta(X):=\inf_{\{B_i\}}\sum_i\diam(B_i)^s\,,
$$
where the infimum is taken over all possible $\delta$-covers  $\{B_i\}$ of $X$. The {\em $s$-dimensional Hausdorff measure} $\cH^s(X)$ of $X$ is defined to be
$$
\cH^s(X):=\lim_{\delta\to0^+}\cH^s_\delta(X)\,
$$
and the {\em Hausdorff dimension} $ \dim  X$ of $X$ by
$$
\dim X  := \inf \{ s: \cH^s(X) =0 \} =\sup  \{s: \cH^s(X) =  \infty \}\,.
$$

The countable collection $\{B_i\}$ is called a \emph{fine cover of $X$} if for every $\delta>0$ it contains a subcollection that is a $\delta$-cover of $X$. The following statement is an immediate and well known consequence of the definition of $\cH^s$.

\begin{lemma}\label{l3}
Let $\{B_i\}$ be a fine cover of $X$ and $s>0$ be such that
$
\sum_{i}\diam(B_i)^s<\infty\,.
$
Then
$$
\cH^s(X)=0.
$$
\end{lemma}

\subsubsection{Proof: the convergence case}

We are given that $\sum_{q=1}^\infty  q^{2-s} \psi(q)^{s-1}<\infty$.  As already mentioned, we can assume that $\bm\theta\in\I^2$.  The proof will make use of the covering of $\cSM_2(\psi;\bm\theta) \cap \I^2$ given by Lemma~\ref{l2}. The rectangle $S_{\bm\theta}(q,m,p_1,p_2)$ arising from this lemma has sides of lengths
$$
A:=\frac{2^{-|m|+1}\sqrt{2\psi(2^t)}}{2^t}\qquad\text{and}\qquad B:=\frac{2^{|m|+1}\sqrt{2\psi(2^t)}}{2^t}
$$
and so can be split into $B/A=2^{2|m|}$ squares with sidelength $A$. By Lemma~\ref{l2}, the collection of such squares taken over $t\ge\ell$ and over $q,p_1,p_2,m$ as specified in the lemma is a $\delta$-cover of $\cSM_2(\psi;\bm\theta) \cap \I^2 $ with $\delta:=\sqrt2A\to0$ as $\ell\to\infty$.
Therefore, the collection of all such squares, say $\{B_i\}$, is a fine cover of $\cSM_2(\psi;\bm\theta) \cap \I^2 $. It follows that
\begin{align}
\sum_{i}\diam(B_i)^s & \ll \ \sum_{t=0}^\infty \ \sum_{2^t\le q< 2^{t+1}} \!\! \sum_{\substack{m\in\Z\\ 2^{|m|}\sqrt{\psi(2^t)}\le 1}}\ \sum_{p_1=-1}^q \ \sum_{p_2=-1}^q
\left(2^{-|m|}\ \frac{\sqrt{\psi(2^t)}}{2^t}\right)^{s} 2^{2|m|} \nonumber \\[1ex]
& \ll \ \sum_{t=0}^\infty\ \sum_{\substack{m\in\Z\\ 2^{|m|}\sqrt{\psi(2^t)}\le 1}} \left(2^{-|m|}\ \frac{\sqrt{\psi(2^t)}}{2^t}\right)^{s} 2^{2|m|}\ 2^{3t} \nonumber \\[1.5ex]
\label{eq2x}& \ll   \ \sum_{t=0}^\infty  \left(\frac{\sqrt{\psi(2^t)}}{2^t}\right)^{s}  2^{3t} \!\!\! \sum_{\substack{m\in\Z\\ 2^{|m|}\sqrt{\psi(2^t)}\le 1}} 2^{(2-s)|m|}\,  .
\end{align}
Since $1<s<2$, the sum over $m\not=0$ in the right hand side of \eqref{eq2x} is a finite increasing geometric progression, which is easily estimated to give
\begin{equation}\label{s1}
\sum_{\substack{m\in\Z\\ 2^{|m|}\sqrt{\psi(2^t)}\le 1}} 2^{(2-s)|m|}
\ll \left(\frac{1}{\sqrt{\psi(2^t)}}\right)^{2-s}=\big(\sqrt{\psi(2^t)}\,\big)^{s-2}\,.
\end{equation}
Substituting this into \eqref{eq2x} gives
\begin{align}
\sum_{i}\diam(B_i)^s& \ll \sum_{t=\ell}^\infty  \left(\frac{\sqrt{\psi(2^t)}}{2^t}\right)^{s}  2^{3t} \ \big(\sqrt{\psi(2^t)}\,\big)^{s-2}\nonumber \\[1ex]
& = \sum_{t=\ell}^\infty  2^{(3-s)t}\psi(2^t)^{s-1}  \ll \sum_{q=1}^\infty  q^{2-s} \psi(q)^{s-1} <\infty\,.
\nonumber
\end{align}
By Lemma~\ref{l3}, $\cH^{s}(\cSM_2(\psi;\bm\theta) \cap \I^2) =0$ and thus the proof of the convergence part is complete.

\subsubsection{Proof: the divergence case}

We are given that $\sum_{q=1}^\infty  q^{2-s} \psi(q)^{s-1}=\infty$. Then, by the inhomogeneous version of Jarn\'ik's theorem \cite{Bugeaud} (see also the remark in \cite[\S12.1]{BDV06}), it follows that $\cH^{s-1}(\cS_1(\psi;\theta_1) \cap \I)=\infty$. The observation that led  to \eqref{vb3} is equally valid in the inhomogeneous  setup; that is to say that
$$
\cS_1(\psi;\theta_1)\times\R   \ \subset  \  \cSM_2(\psi;\bm\theta).
$$
Thus,  $\cH^{s}\big(\cSM_2(\psi;\bm\theta) \cap \I^2\big)\ge \cH^{s}\big((\cS_1(\psi;\theta_1) \cap \I) \times\I\big)$.
Since $\cH^{s-1}(\cS_1(\psi;\theta_1) \cap \I )=\infty$,  the  slicing lemma \cite[Lemma~4]{BV06Sliching} implies that
$$
\cH^{s}\big((\cS_1(\psi;\theta_1) \cap \I) \times\I\big)=\infty\,.
$$
Hence $\cH^{s}\big(\cSM_2(\psi;\bm\theta) \cap \I^2\big)=\infty$ and the proof of Theorem~\ref{t1} is complete.

\section{Diophantine approximation on planar curves}

When the coordinates of the approximated point $\vv x \in \R^n$ are confined by functional relations, we fall into the theory  of Diophantine approximation on manifolds \cite{BD99}. Over the last decade or so, the theory of Diophantine approximation on manifolds has developed at  some considerable pace with the catalyst being the pioneering work of Kleinbock \& Margulis.  For details of this and an overview of the pretty much complete results regarding $\cS_n(\psi)$  restricted to manifolds $\cM \subset \R^n$ see \cite{SDA, BDV06} and references within. However, much less is  known regarding  multiplicative Diophantine approximation on manifolds.  It would be highly desirable to address this imbalance by investigating the following analogue of Problem~1 for manifolds.

\medskip

\textbf{Problem 2:} \textit{Determine the Hausdorff measure $\cH^s$ of $\cSM_n(\psi) \cap  \cM$.}

\medskip

\noindent Our goal in this paper  is to  consider  the  problem  in the case $\cM$ is a planar curve $\cC$  and  $\cH^s$ is a  genuine  fractal  measure.

\begin{theorem}\label{t2}
Let $\psi$ be any approximating function, $\bm\theta=(\theta_1,\theta_2)\in\R^2$ and $s \in (0,1)$.   Let  $\cC$ be a $C^{(3)}$ curve in $\R^2$ with non-zero curvature everywhere apart from a set of $s$-dimensional Hausdorff measure zero. Then
\begin{equation}\label{e:006}
\cH^s\big(\cS^\times_2(\psi;\bm\theta) \cap \cC \big) \, = \,
\left\{\begin{array}{ll}
0  & {\rm if} \;\;\;
\textstyle{\sum_{q=1}^\infty} \; q^{1-s}\psi^s(q) \; <\infty  ,\\[2ex]
\infty & {\rm if} \;\;\;
\textstyle{ \sum_{q=1}^\infty } \;  q^{1-s}\psi^s(q) \;
 =\infty  .
\end{array}\right.
\end{equation}
\end{theorem}

\medskip

\noindent {\em Remark 2.1.}  In \cite{BL}, the authors prove  that $\cH^s\big( \cS^\times_2(\psi) \cap \cC\big)=0$ under the more restrictive assumption that $\sum_{q=1}^\infty \; q^{1-s}\psi^s(q)(\log q)^s \; <\infty$. Although not an error, in their  proof of this homogeneous statement \cite[Theorem~1]{BL} there a certain degree of ambiguity in the manner in which  a key counting estimate originating from \cite[Theorem 1]{VV} is  applied.   More precisely,  it is important to stress that the implied constant appearing in inequality  (13) associated with \cite[Theorem VV]{BL}  is independent of $\psi$. This is crucial  as it is applied over a countable family of functions $\psi(Q)$ that depend on a parameter $m \in \Z$.

\subsection{Proof of Theorem \ref{t2}}

\subsubsection{Rational points near planar curves}

The proof of Theorem~\ref{t2} relies on the results obtained in  \cite{BVV11} regarding the distribution of `shifted'
rational points near planar  curves, which we recall here.
In view of the metrical nature of Theorem~\ref{t2}, there is no loss of generality in
assuming that $\cC:=\big\{(x,f(x)):x\in I\big\}$  is the graph of a $C^{(3)}$ function $f :I\to\R$  defined on a finite closed interval  $I$  and that $f''$ is continuous and non-vanishing on $I$.  By the compactness of $I$, there exists positive  and finite constants $c_1,c_2$ such that
\begin{equation}\label{e:002x}
 c_1 \ \le \  |f''(x)| \ \le \ c_2\qquad\text{for all } \ x\in I\,.
\end{equation}

Given  $\bm\theta =(\theta_1,\theta_2)\in \R^2$,
$\delta > 0 $ and $Q\ge 1$, consider  the set
\begin{equation*}
A_{\bm\theta}(Q,\delta) \ := \ \left\{ (p_1,q) \in \Z \times {\mathbb N} \, : \,
\begin{array}{l}
 Q < q\le 2Q,\ (p_1+\theta_1)/q\in I\\[1ex]
 \|qf(\,(p_1+\theta_1)/q\,)-\theta_2\|<\delta
 \end{array}
 \right\} .
\end{equation*}
The function $N_{\bm\theta}(Q,\delta)=\#A_{\bm\theta}(Q,\delta)$ counts the number of rational
points $(p_1/q,p_2/q)$ with bounded denominator $q$ such that the
shifted points $((p_1+\theta_1)/q,(p_2+\theta_2)/q)$  lie
within the $\delta/Q$-neighborhood of the
curve $\mathcal{C}$. The following result is a direct consequence of Theorem~3 from \cite{BVV11} -- the inhomogeneous generalization of  Theorem 1 from \cite{VV}.

\begin{theorem}\label{BVV}
\label{t:thm1}
Let $f\in C^{(3)}(I)$ and satisfy \eqref{e:002x} and let $\ve>0$. Then, for any $Q\ge 1$ and $0<\delta\le\frac12$ we have that
\begin{equation*}
N_{\bm\theta}(Q,\delta) \ll \delta Q^2 + Q^{1+\ve}
\end{equation*}
where the implied constant is independent of $Q$ and $\delta$.
\end{theorem}

\subsubsection{Proof: the convergence case}

We are given that $\sum_{q=1}^\infty  q^{1-s} \psi(q)^{s}<\infty$. Since $\psi$ is monotonic, we have that
\begin{equation}\label{x11}
\sum_{t=1}^\infty  2^{(2-s)t} \psi(2^t)^{s}<\infty.
\end{equation}
Hence
\begin{equation}\label{x10}
2^{(2-s)t}\psi(2^t)^s<1\qquad\text{ for all sufficiently large $t$.}
\end{equation}

As in the proof of Theorem \ref{t1},  without loss of generality we assume that $\bm\theta=(\theta_1,\theta_2)\in\I^2$ and moreover that $\cC\subset \I^2$.
By Lemma~\ref{l2}, for any $\ell\in\N$ we have that
\begin{equation}\label{eS:027}
\cSM_2(\psi;\bm\theta) \cap \cC\ \subset   \  \bigcup_{t=\ell}^\infty\
\bigcup_{2^t\le q< 2^{t+1}}\
\bigcup_{\substack{m\in\Z\\[0.5ex] 2^{|m|}\sqrt{\psi(2^t)}\le 1}} \ \bigcup_{p_1=-1}^q \ \bigcup_{p_2=-1}^q \cC\cap S_{\bm\theta}(q,m,p_1,p_2) \,.
\end{equation}
It is easily verified that
\begin{equation}\label{eS:028}
\diam(\cC\cap S_{\bm\theta}(q,p_1,p_2,m)) \ \ll \  2^{-|m|}\
\frac{\sqrt{\psi(2^t)}}{2^t}
\end{equation}
and that $(p_1,q)\in A_{\bm\theta}(Q,\delta)$ with
$$
\delta  \ \asymp   \  2^{|m|}\, \sqrt{\psi(2^t)} \qquad\text{and}\qquad Q=2^t \,.
$$
By \eqref{eS:027}, the collection of all such  sets $\cC\cap S_{\bm\theta}(q,m,p_1,p_2)$ is  a fine cover of $\cSM_2(\psi;\bm\theta)\cap\cC$.
By Theorem~\ref{BVV} with $\ve:=s/4$, it follows that
$$
N_{\bm\theta}(Q,\delta) \ \ll  \ 2^{|m|}\, \sqrt{\psi(2^t)}\, 2^{2t} + 2^{(1+s/4)t} \,
$$
and so the  $s$-dimensional volume of the above  fine  cover  is
\begin{align}
 & \ll \ \ \sum_{t=1}^\infty\ \sum_{\substack{m\in\Z\\[0ex] 2^{|m|}\sqrt{\psi(2^t)}\le 1}} \hspace*{-2ex}\left(2^{-|m|}\ \frac{\sqrt{\psi(2^t)}}{2^t}\right)^{\!\!\!s} \left(2^{|m|}\ \frac{\sqrt{\psi(2^t)}}{2^t}\ 2^{3t} + 2^{(1+s/4)t}\right)\nonumber \\[2ex]
\label{eq}& \ll  \ \  \sum_{t=1}^\infty  \left(\frac{\sqrt{\psi(2^t)}}{2^t}\right)^{s+1}  2^{3t} \ \sum_{\substack{m\in\Z\\[0.5ex] 2^{|m|}\sqrt{\psi(2^t)}\le 1}} 2^{(1-s)|m|}\  \quad +\nonumber \\[1ex]
&
\hspace*{20ex}+  \quad \sum_{t=1}^\infty\ \sum_{m\in\Z} \left(2^{-|m|}\ \frac{\sqrt{\psi(2^t)}}{2^t}\right)^{\!\!\!s} 2^{(1+s/4)t}\nonumber \\[2ex]
& \stackrel{\eqref{s1}}{\ll}  \ \ \sum_{t=1}^\infty  \left(\frac{\sqrt{\psi(2^t)}}{2^t}\right)^{s+1}  2^{3t} \ \big(\sqrt{\psi(2^t)}\,\big)^{s-1} \ \quad +\nonumber \\[1ex]
&
\hspace*{20ex}+ \quad \sum_{t=1}^\infty\ \left(\frac{\sqrt{\psi(2^t)}}{2^t}\right)^{\!\!\!s} 2^{(1+s/4)t}\nonumber \\[2ex]
& \stackrel{\eqref{x10}}{\ll} \  \ \sum_{t=1}^\infty  2^{(2-s)t}\psi(2^t)^{s}
+\sum_{t=1}^\infty\ 2^{-ts/4} \  \  \stackrel{\eqref{x11}}{<} \  \ \infty\,.
\nonumber
\end{align}
By Lemma~\ref{l3}, $\cH^{s}(\cC\cap\cSM_2(\psi;\bm\theta)) =0$ and the proof of the convergence part of Theorem~\ref{t2} is complete.

\subsubsection{Proof: the divergence case}

We are given that $\sum_{q=1}^\infty  q^{1-s} \psi(q)^{s}=\infty$. Then, by the inhomogeneous version of Jarn\'ik's theorem \cite{Bugeaud} (see also the remark in \cite[\S12.1]{BDV06}), we have that $\cH^{s}(\cS_1(\psi;\theta_1)\cap I)=\infty$.  The  same observation that led  to \eqref{vb3},   gives rise to  the following obvious inclusion
$$
X:=\big\{(x,f(x)): x\in \cS_1(\psi;\theta_1)\cap I\big\}  \ \subset  \ \cC\cap \cSM_2(\psi;\bm\theta).
$$
Since $f\in C^{(1)}$,  we have that $f$ is locally bi-Lipshitz and thus the map $x\mapsto(x,f(x))$ preserves $s$-dimensional Hausdorff measure. Therefore,
$$
\cH^s\big(\cSM_2(\psi;\bm\theta) \cap \cC \big)\ge\cH^s(X)=\cH^{s}(\cS_1(\psi;\theta_1)\cap I)=\infty\,
$$
and so completes the  proof of Theorem~\ref{t2}.

\section{Inhomogeneous badly approximable points}

A real number $x$ is said to be {\em badly approximable} if there
exists a positive constant $c(x)$ such that
\begin{equation*}
 \|qx\| \ > \ c(x) \ q^{-1}  \quad \forall \  q \in \N   \ .
\end{equation*}
Here and throughout $ \| \cdot  \| $ denotes the distance of a real
number to the nearest integer. It is well know that the set $\bad$
of badly approximable numbers is of Lebesgue measure zero. However,
a result of Jarn\'{\i}k  \cite{Jarnik-bad} states that
\begin{equation}
\label{badfulldim} \dim \bad = 1  \ .
\end{equation}
Thus, in terms of dimension the set of badly
approximable numbers is maximal; it has the same dimension as the
real line.

In higher dimensions there are various natural generalizations of
$\bad$. Restricting our attention to the plane $\R^2$, given a pair
of real numbers $i$ and $j$ such that
\begin{equation}\label{neq1}
0\le i,j\le 1   \quad {\rm and \  } \quad  i+j=1  \, ,
\end{equation}
a point $(x_1,x_2) \in \R^2$ is said to be {\em $(i,j)$-badly
approximable} if there exists a positive constant $c(x_1,x_2)$ such that
\begin{equation*}
  \max \{ \; \|qx_1\|^{1/i} \; , \ \|qx_2\|^{1/j} \,  \} \ > \
  c(x_1,x_2) \ q^{-1} \quad \forall \  q \in \N   \ .
\end{equation*}
Denote by   $\bad(i,j)$  the set of $(i,j)$-badly approximable
points in $\R^2$.  If $i=0$, then we use the convention that $ x^{1/i}\::=0 $ and so $\bad(0,1)$ is identified with
$\R \times \bad $. That is, $\bad(0,1)$ consists of points $(x_1,x_2)$
with $x_1 \in \R$ and $x_2 \in \bad$.  The roles of $x_1$ and $x_2$ are
reversed if $j=0$. It easily follows from classical results in the
theory of metric Diophantine approximation   that $\bad(i,j)$ is of
(two-dimensional) Lebesgue measure zero. Building upon the work of Davenport \cite{Davenport-64:MR0166154}, it is shown  in \cite{PV}  that
\begin{equation} \label{pollvel}
\dim \bad(i,j) =  2 \, .
\end{equation}
For alternative proofs and various strengthenings  see \cite{Badziahin-Pollington-Velani-Schmidt,  Fishman-09:MR2528057, Kleinbock-Weiss-05:MR2191212, Kleinbock-Weiss-10:MR2581371,KTV}.  In particular, a consequence of the main result in \cite{Badziahin-Pollington-Velani-Schmidt} is that the intersection of any finite number of $(i,j)$-badly approximable sets is of full dimension. Obviously this implies that the intersection of any two such sets is non-empty and thus establishes a conjecture of Wolfgang Schmidt dating back to the eighties. Most recently,  Jinpeng An \cite{An} has shown that the set $\bad(i,j)$ is winning in the sense of Schmidt games and thus the intersection of any countable number of $(i,j)$-badly approximable sets is of full dimension.

\noindent The goal  in this paper is to obtain the analogue of (\ref{pollvel})  within  the   inhomogeneous  setup.

\medskip

\textbf{Problem 3:} \textit{Find an analogue of \eqref{pollvel} for inhomogeneous approximation.}

\medskip

\noindent For
$\bm\theta=(\theta_1,\theta_2)\in\R^2$, let $\bad(i,j;\bm\theta)$ denote the set of points  $ (x_1,x_2) \in \R^2 $ such that
\begin{equation*}
  \max \{ \; \|qx_1 - \theta_1 \|^{1/i} \; , \ \|qx_2 - \theta_2 \|^{1/j} \,  \} \ > \
  c(x_1,x_2) \ q^{-1} \quad \forall \  q \in \N   \ .
\end{equation*}
Naturally, given $\theta \in \R$ the inhomogeneous generalisation of the one-dimensional set $\bad$ is the set
$$
\bad(\theta)  := \{ x \in \R :  \exists \ c(x)>0  \ { \rm \ so \  that  \ }
 \|qx - \theta \| \ > \ c(x) \ q^{-1}  \quad \forall \  q \in \N  \}  \
$$
and so, for example,  $\bad(0,1; \theta)$ is identified with
$\R \times \bad(\theta) $.
It is straight forward to deduce that $\bad(i,j;\bm\theta)$ is of measure zero from the inhomogeneous version of Khintchine's theorems with varying  approximating functions in each co-ordinate.  We will prove the  following full dimension statement which represents  the inhomogeneous analogue of (\ref{pollvel}).

\begin{theorem}\label{t3} Let $i,j$ satisfy (\ref{neq1}) and $\bm\theta=(\theta_1,\theta_2)\in\R^2$.  Then  $$
\dim \bad(i,j;\bm\theta)  = 2 \, .$$
\end{theorem}

The basic philosophy behind the proof is simple and is likely to be applicable to other situations where  the goal is to generalize a known  homogenous badly approximable statement to the inhomogeneous setting -- see Remark 3.5 below.  The key is to  exploit  the known homogeneous `intervals construction' proof and use the power of subtraction; namely
$$
\mbox{(homogeneous construction)} \quad   +  \quad   (\bm\theta - \bm\theta = \bm0)    \quad  \Longrightarrow    \quad \mbox{(inhomogeneous statement).}
$$

\medskip

\noindent Before moving onto the proof of Theorem \ref{t3}, several remarks are in order.

\medskip

\noindent {\em Remark 3.1.}  For  $i$ and $j$ fixed, the proof can be easily modified to deduce that the intersection of any finite number of $\bad(i,j;\bm\theta)$ sets is of full dimension.  In fact, by making use of standard trickery (such as the argument that proves that the countable intersection of winning sets is winning)  one can actually deduce that for any countable sequence $\bm\theta_t \in \R^2$
$$
\dim  \big(\cap_{t=1}^{\infty} \bad(i,j;\bm\theta_t)  \big)  = 2 \, .
$$

\noindent {\em Remark 3.2.}  In another direction, the proof can be adapted to obtain the following more general form of Theorem \ref{t3} in which the inhomogeneous factor $\bm\theta$ depends on $(x_1,x_2)$. More precisely, let
$\bm\theta=(\theta_1,\theta_2):\R^2\to\R^2$ and let $\bad(i,j;\bm\theta)$ denote the set of points  $ (x_1,x_2)$ such that
\begin{equation*}
  \max \{ \; \|qx_1 - \theta_1(x_1,x_2) \|^{1/i} \; , \ \|qx_2 - \theta_2(x_1,x_2) \|^{1/j} \,  \} \ > \
  c(x_1,x_2) \ q^{-1} \quad \forall \  q \in \N   \ .
\end{equation*}
Then, if $\theta_1=\theta_1(x_1)$ and $\theta_2=\theta_2(x_2)$ are Lipshitz functions of one variable, we have that
$$
\dim \bad(i,j;\bm\theta)  =  2 \, .
$$
 As an example,  this statement implies that  there is a set of $(x_1,x_2) \in \R^2 $ of Hausdorff dimension 2 such that
\begin{equation*}
  \max \{ \; \|qx_1 - x_1^2 \|^{1/i} \; , \ \|qx_2 - x_2^3 \|^{1/j} \,  \} \ > \
  c(x_1,x_2) \ q^{-1} \quad \forall \  q \in \N   \ .
\end{equation*}
It is worth pointing out that in the case $i=j$, the   statement is also true if $\theta_1=\theta_1(x_1,x_2)$ and $\theta_2=\theta_2(x_1,x_2)$ are Lipshitz functions of two variables.

\noindent {\em Remark 3.3.}  There is no difficulty  in  establishing  the  higher dimension analogue of Theorem~\ref{t3}.  For any $\bm\theta = (\theta_1, \ldots, \theta_n)  \in \R^n$ and  $n$--tuple of real numbers $i_1, ...,i_n \geq 0 $ such
that $\sum i_r = 1 $, denote by $\bad(i_1, \ldots,i_n; \bm\theta)$ the set of
points $(x_1, ...,x_n) \in \R^n $ for which  there exists a
positive constant $ c(x_1, ...,x_n)$ such that   $$ \max \{ \;
||qx_1  - \theta_1||^{1/i_1} \; , ..., \ ||qx_n   -   \theta_n||^{1/i_n} \,  \} \ > \
c(x_1, ...,x_n) \ q^{-1} \ \ \ \forall \ \ \ q \in \N . $$ By
modifying the proof of Theorem \ref{t3}, in the obvious way,  it is easy to show that $$
\dim \bad(i_1 \ldots,i_n;\bm\theta)  = n  \, .$$
Moreover, the various proofs of the  homogeneous results obtained in \cite{KTV} regarding $ \bad(i_1 \ldots,i_n) \cap \Omega $,
 where $ \Omega $ is some `nice' fractal set  (essentially, the support set of an absolutely friendly, Ahlfors regular measure)  can be adapted to give the corresponding  inhomogeneous  statements  without any serious difficulty.  We have decided to restrict ourselves to proving simply Theorem \ref{t3} since it already contains the necessary ingredients  to obtain the inhomogeneous statement from the homogeneous proof.

 \medskip

\noindent {\em Remark 3.4.} In the symmetric case $ i_1=\ldots=i_n=1/n$,  our Theorem \ref{t3} and indeed its generalizations mentioned in the previous remark are  covered by  the work of Einsiedler $\&$  Tseng  \cite[Theorem 1.1]{ET}.  They actually deal with badly approximable systems of linear forms and show that the intersection of such sets with the support set $\Omega$  of an absolutely friendly measure is winning in the sense of Schmidt games.  We mention in passing, that  Einsiedler $\&$  Tseng proved their results roughly at the same time as us, but for some mystical reason, it has taken us over four years to present our work.  Indeed the second author had a useful discussion with Einsiedler regarding their preprint at the conference `The Diverse Faces of Arithmetic'   in honour of the late Graham Everest in 2009.

\medskip

\noindent {\em Remark 3.5.}  The basic philosophy behind the proof of Theorem \ref{t3} can be exploited to yield the inhomogeneous  strengthening  of Schmidt's Conjecture.  More precisely, we are able show that any inhomogeneous $\bad(i,j;\bm\theta)$ set is winning  and thus
$$
\dim  \big(\cap_{t=1}^{\infty} \bad(i_t,j_t;\bm\theta_t)  \big)  = 2 \, .
$$
Furthermore, it is possible to show that the intersection of $\bad(i,j;\bm\theta)$ with any non-degenerate planar curve $\cC$ is winning as is the intersection with any straight line satisfying certain natural Diophantine conditions. The former implies that
$$
\dim  \big( \cap_{t=1}^{\infty} \bad(i_t,j_t;\bm\theta_t) \cap  \cC \big)  = 1 \, ,
$$
which strengthens even the homogeneous results obtained in \cite{DavBV,DavVB} that solve an old problem of Davenport. These winning results will be the subject of a forthcoming joint paper with Jinpeng An.

\subsection{Proof of Theorem \ref{t3}}

Throughout, we fix $i,j > 0 $ satisfy (\ref{neq1}) and $\bm\theta=(\theta_1,\theta_2)\in\R^2$. The situation when either $i=0$ or  $j=0$ is easier and will be  omitted.

Since   $ \bad(i,j;\bm\theta)  \subseteq \R^2$, we  obtain for free  the upperbound result:
 $$
\dim \bad(i,j;\bm\theta)  \le 2 \, .$$
Thus, the proof reduces to establishing the complementary lowerbound.   With this in mind, for a fixed constant $c > 0$
let $$
\bad_c(i,j;\bm\theta) := \{ (x_1,x_2) \in \R^2 :  \max \{  \|qx_1 - \theta_1 \|^{1/i} ,  \|qx_2 - \theta_2 \|^{1/j}   \} \, > \,
  c/q \quad \forall \,  q \in \N  \} \, . $$

\noindent Clearly $\bad_c(i,j;\bm\theta) \subset \bad(i,j;\bm\theta)$ and  $$\bad(i,j;\bm\theta) = \bigcup_{c >0} \bad_c(i,j;\bm\theta) \, . $$ Geometrically,
the set $\bad_c(i,j;\bm\theta)$ simply consists of points $(x_1,x_2) \in \R^2$
which avoid all rectangles centred at inhom-rational points  $((\theta_1 -p_1)/q,(\theta_2-p_2)/q)$
of side length $2c^i q^{-(1+i)} \times 2c^j q^{-(1+j)}$. The sides
are taken to be parallel to the coordinate axes. The overall strategy is to construct a  `Cantor--type' subset
$\K (= \K(i,j)) $ of $\bad_c(i,j;\bm\theta)$ with the property that $ \dim \K  \to 2 $ as $ c \to 0$. This together with the fact that
$$ \dim \bad(i,j;\bm\theta) \ \ge \  \dim \bad_c(i,j;\bm\theta) \ \ge \
\dim \K  $$  implies the  required lower bound result.

To obtain the desired Cantor type set  $\K$,  we  adapt  the homogeneous construction of $\hK =\hK^{\!\!\!\bm0} $ given in \cite[\S3.1]{PV} that is at the heart of establishing $(\ref{pollvel})$; that is to say  Theorem~\ref{t3}  with $\bm\theta = \bm0$.

\subsubsection{The homogeneous construction}
Let  $R  \ge 11  $ be  an integer  and  $c > 0 $ be given by
\begin{equation} c \   :=  \ 8^{-1/i} \, R^{-2(1+i)/i}  \, .  \label{c}
\end{equation}
It is established  in \cite[\S3.1]{PV}, by induction on $ n \ge 0$, the existence of a nested collection
 $\cF_n$ of closed rectangles $F_n := I_n  \times J_n $ with the property  that
for all points $(x_1, x_2) \in F_n$  the following (homogeneous)   condition
is  satisfied:
$$ ~ \hspace{2cm}
  \max \{ \; \|qx_1 \|^{1/i} \; , \ \|qx_2 \|^{1/j} \,  \} \ > \
  c  \, q^{-1}
\hspace{1cm} \forall \ \  \ 0 \ < \
q \ <  \  R^n \, . \hspace{12mm} {\rm (H)} $$
The side lengths $I_n$ and $J_n$ of $F_n$ are  given by
\begin{equation}
 I_n \; := \; \mbox{\small{$\frac{1}{4}$}} \, R^{-(1+i)(n+1)}
\hspace{0.7cm} {\rm and } \ \hspace{0.7cm} J_n \; := \; \mbox{\small{$\frac{1}{4}$}} \,
R^{-(1+j)(n+1)}  \qquad (n \ge 0) \, .  \label{size}
\end{equation}

\noindent
Without
loss of generality assume that $0< i \le j < 1$ so that the
rectangles $F_n$ are long and thin unless $i=j$ in which case the
rectangles are obviously squares.

 The crux of the induction is as follows.
We work within the
closed unit square and start by subdividing the square into closed
rectangles $F_0$ of size $I_0 \times J_0$  -- starting from the
bottom left hand corner of the unit square (i.e. the origin).  Denote by $\cF_0$ the collection of
rectangles $F_0$. For $n=0$, condition
(H)  is trivially satisfied for any rectangle $F_0 \subset
\cF_0$, since there are no integers $q$ satisfying $ 0<q<1$. Given
$\cF_n$ satisfying condition (H), we wish to construct a
nested collection $\cF_{n+1}$ for which the condition is
satisfied for $n+1$.  Suppose $F_n$ is a good rectangle; that is,
all points $(x_1,x_2) \in F_n$ satisfy condition (H). In
short, $F_n \in \cF_n$.  Now partition $F_n$ into rectangles $F_{n+1}$ of size $I_{n+1}
\times J_{n+1} $ -- starting from the bottom left hand corner of
$F_n$. From (\ref{size}), it follows that there are $[R^{1+i}]
\times [R^{1+j}]$ rectangles in the partition.  Since they are nested,  anyone of these rectangles will satisfy condition (H) for $n+1$
if for any point $(x_1,x_2)$
in $F_{n+1}$ the inequality
\begin{equation}
\max \{ \; ||qx_1||^{1/i} \; , \ ||qx_2||^{1/j} \,  \} \ > \  c \;
q^{-1}
\label{ineqij}
\end{equation}
is satisfied for \begin{equation}
 R^n \ \le \ q \ < R^{n+1} \ .
\label{rangeij}
\end{equation}
With $q$ in this `denominator' range, suppose there exists a bad
rational pair $(p_1/q,p_2/q)$ so that (\ref{ineqij}) is violated, in
other words $$ |x_1 - p_1/q | \le c^i \, q^{-(1+i)} \hspace{1cm}
{\rm and } \hspace{1cm}
 |x_2 - p_2/q| \le c^j \, q^{-(1+j)} \  $$
 for some point  in $F_n$ and therefore in some $F_{n+1}$. Such    $F_{n+1}$ rectangles are bad  in the sense that  they do not satisfy condition (H) for $n+1$ and those that remain are good.  The upshot of the `Stage 1'  argument in \cite[\S3.1]{PV} is that there are at most
 \begin{equation}
   3 \, \left[ \frac{J_n}{J_{n+1}} \right]  \, \le \, 3 R^{1+j}
 \end{equation}
 bad $F_{n+1}$ rectangles in $F_n$.  Hence, out of the potential  $[R^{1+i}]
\times [R^{1+j}]$ rectangles,   at least
\begin{eqnarray*}
 (R^{1+i} -1)(R^{1+j} -1) \; - \; 3 R^{1+j}  \ > \   R^3(1- 5 R^{-(1+i)})  \,
 \label{card1}
\end{eqnarray*}
are good  $F_{n+1} $ rectangles in $F_n$.  Now choose exactly $ [R^3 \, ( 1
\; - \; 5 \, R^{-(1+i)} )]$ of these good rectangles and denote this collection by ${\cal F}(F_n)$. Finally, define
 $$ \cF_{n+1} \ := \ \bigcup_{F_n \subset \cF_n } {\cal
F}(F_n) \ . $$  Thus, given the collection $\cF_n$ for which
condition (H) is  satisfied for $n$, we have constructed a
nested collection $\cF_{n+1}$ for which  condition (H) is
satisfied for $n+1$. This completes the proof of the induction
step and so the construction of the Cantor-type set $$ \hK \ := \
\bigcap_{n=0}^{\infty} \cF_n \, . $$

\medskip

\subsubsection{Bringing the inhomogeneous approximation into play}

The idea  is to  merge the inhomogeneous approximation constraints into the above homogeneous construction. In short, this involves  creating  a subcollection $\cF_n^{\bm\theta}$ of  $\cF_n$ so that for all points $(x_1, x_2) \in F_n$ with $F_n  \subseteq \cF_n^{\bm\theta}$, both the  (homogeneous)   condition (H)  and the  following (inhomogeneous)  condition
are   satisfied:
$$ ~ \hspace{1cm}
  \max \{ \; \|qx_1 - \theta_1 \|^{1/i} \; , \ \|qx_2 - \theta_2\|^{1/j} \,  \} \ > \
c_*   \, q^{-1}
\hspace{1cm} \forall  \  \ 0 \ < \
q \ <  \  R^{n-d} \, , \hspace{12mm} {\rm (I)} $$
where
\begin{equation}
c_*^i \, := \,  \mbox{\small{$\frac{1}{8}$}}  R^{-(1+i)(d+2)} \qquad   {\rm and} \qquad d:= \left[\frac{3}{i}\right]
\label{c_*}
\end{equation}

\noindent For $n=0$, condition
(I)  is trivially satisfied for any rectangle $F_0 \subset
\cF_0$, since there are no integers $q$ satisfying $ 0<q<1$. Put $\cF_0^{\bm\theta} := \cF_0$.  Now suppose $\cF_n^{\bm\theta}  \subseteq  \cF_n$ has been constructed  and for each $F_n$  in  $\cF_n^{\bm\theta}$ construct the collection $\cF (F_n) $ as  before.  Then by definition, each $  F_{n+1} \in \cF (F_n) $ satisfies condition  (H)  for $n+1$.  The aim is to construct a subcollection  $ \cF^{\bm\theta} (F_n) $ such that for each $  F_{n+1} $ in $ \cF^{\bm\theta} (F_n) $ condition (I) for $n+1$  is also satisfied; in other words, for any point $(x_1,x_2)$
in $F_{n+1}$ the inequality
\begin{equation}
\max \{ \; \|qx_1 - \theta_1\|^{1/i} \; , \ \|qx_2 -\theta_2 \|^{1/j} \,  \} \ > \
c_*   \, q^{-1}
\label{hineqij}
\end{equation}
is satisfied for \begin{equation}
 R^{n-d} \ \le \ q \ < R^{n+1-d} \ .
\label{hrangeij}
\end{equation}
With $q$ satisfying \eqref{hrangeij}, suppose there exists a bad
inhom-rational pair $((\theta_1 +p_1)/q,(\theta_2+p_2)/q)$ so that (\ref{hineqij}) is violated, in
other words
 $$ |x_1 - (\theta_1 + p_1)/q | \le c_*^i \, q^{-(1+i)}   \hspace{1cm}
{\rm and } \hspace{1cm}
 |x_2 - (\theta_2 + p_2)/q| \le c_*^j \, q^{-(1+j)}  \  $$
for some point  $(x_1,x_2)$ in $F_n$.
Then, in view of   \eqref{hrangeij}, it follows that
\begin{equation}
 |x_1 - (\theta_1 + p_1)/q |   \ \le  \    c_*^i \, R^{-(1+i)(n-d)}  \ \le   \  \mbox{\small{$\frac{1}{2}$}} \, |I_{n+1}|
 \label{sv1}
\end{equation}
if
\begin{equation}
  c_*^i  \ \le \ \mbox{\small{$\frac{1}{8}$}}  R^{-(1+i)(d+2)}  \, .
 \label{ssv1}
\end{equation}
\noindent Similarly,
\begin{equation}
 |x_2 - (\theta_2 + p_2)/q| \  \le   \  \mbox{\small{$\frac{1}{2}$}} \, |J_{n+1}|
 \label{sv2}
\end{equation}
if
\begin{equation}
  c_*^j  \ \le \ \mbox{\small{$\frac{1}{8}$}}  R^{-(1+j)(d+2)}  \, .
 \label{ssv2}
\end{equation}
Observe that \eqref{ssv1} implies \eqref{ssv2}.  In view of \eqref{c_*}, we have equality in  \eqref{ssv1}, thus \eqref{sv1} and \eqref{sv2} are satisfied  and it follows that  any  bad inhom-rational pair gives rise to at most 4 bad rectangles $F_{n+1} $ in $  \cF (F_n) $; i.e. rectangles for which (I) is not satisfied for $n+1$.  Now suppose there exist two bad inhom-rational pairs, say   $((\theta_1 +p_1)/q,(\theta_2+p_2)/q)$  and $((\theta_1 + \tilde{p}_1)/\tilde{q},(\theta_2-\tilde{p}_2)/\tilde{q})$.  Then, for any   $(x_1,x_2)  \in F_n$ we have that
\begin{eqnarray*}
 |x_1 - (\theta_1 + p_1)/q |    & \stackrel{\eqref{sv1}}{\le}  &     \mbox{\small{$\frac{1}{2}$}} \, |I_{n+1}|  +   |I_n|  \  <  \ 2  |I_n| \nonumber  \\[2ex]
~ \Longrightarrow & &  \!\!\!\!  | q x_1 - \theta_1 - p_1 |   \  <    \   q \, 2  |I_n|  \ \stackrel{\eqref{hrangeij}}{\le}   \ 2   R^{n+1-d} |I_n| \,  ~
\end{eqnarray*}
and
\begin{eqnarray*}
 |x_2 - (\theta_2 + p_2)/q |    & \stackrel{\eqref{sv2}}{\le}  &    \ \mbox{\small{$\frac{1}{2}$}} \, |J_{n+1}|  +   |J_n|   \ <   \  2  |J_n| \nonumber  \\[2ex]
~ \Longrightarrow & &  \!\!\!\!  | q x_2 - \theta_2 - p_2 |   \ <    \   q \, 2  |J_n|  \ \stackrel{\eqref{hrangeij}}{\le}   \ 2   R^{n+1-d} |J_n| \, .
\end{eqnarray*}
Similarly, we obtain that
\begin{equation*}
 | \tilde{q} x_1 - \theta_1 - \tilde{p}_1 |    <       2   R^{n+1-d} |I_n|   \qquad {\rm and } \qquad  | \tilde{q} x_2 - \theta_2 - \tilde{p}_2) |   <       2   R^{n+1-d} |J_n| \, .
\end{equation*}
Let $q_* := | q - \tilde{q} |  $ and observe that
\begin{equation}
 0<  q_* < R^{n+1-d} < R^n
 \label{Rm}
\end{equation}
It now follows  that
\begin{eqnarray}
|q_*  x_1 - (p_1 + \tilde{p}_1) |  & = &  | (q x_1 - \theta_1 - p_1   ) - ( \tilde{q} x_1 - \theta_1 - \tilde{p}_1  ) |  \nonumber \\[2ex]   & \le & | q x_1 - \theta_1 - p_1   | +   |\tilde{q} x_1 - \theta_1 - \tilde{p}_1 |    \nonumber \\[2ex]
&  \le &  4 R^{m+1} |I_n| \  \le  \   \mbox{\small{$\frac{1}{2}$}}  \,  R^{-d(1+i)} q_*^{-i} \,   \stackrel{\eqref{c_*}}{\le}  \,  R^{-3(1+i)} q_*^{-i} \,    \nonumber \\[2ex]   & \stackrel{\eqref{c}}{\le}  & c^i  q_*^{-i}  \, .
\label{tut1}
\end{eqnarray}
Similarly,
\begin{eqnarray}
|q_*  x_2 - (p_2 + \tilde{p}_2) |  & = &  | (q x_2 - \theta_2 - p_2   ) - ( \tilde{q} x_2 - \theta_2 - \tilde{p}_2  ) |  \nonumber \\[2ex]   & \le & | q x_2 - \theta_2 - p_2   | +   |\tilde{q} x_2 - \theta_1 - \tilde{p}_2 |    \nonumber \\[2ex]
&  \le &  4 R^{m+1} |J_n| \  \le  \   \mbox{\small{$\frac{1}{2}$}}  \,  R^{-d(1+j)} q_*^{-j} \,   \stackrel{\eqref{c_*}}{\le}  \,   \mbox{\small{$\frac{1}{2}$}}  \, R^{-\frac{(3-i)}{i}(1+j)} q_*^{-j} \,    \nonumber \\[2ex]   & \stackrel{\eqref{c}}{\le}  & c^j  q_*^{-j}  \, .
\label{tut2}
\end{eqnarray}

The upshot of  inequalities \eqref{Rm}, \eqref{tut1}  and \eqref{tut2}  is that the homogeneous condition (H) is not satisfied for points in $F_n$.  This contradicts the fact that
$F_n \in \cF_n^{\bm\theta} \subseteq  \cF_n  $.
In turn, this implies that there exists at most one bad inhom-rational pair  that gives rise to at most 4 bad  $F_{n+1}$ rectangles amongst those in
$  \cF(F_n) $.  In other words, at least
\begin{eqnarray*}
\# \cF(F_n) - 4    \ = \ [R^3 \, ( 1
\; - \; 5 \, R^{-(1+i)} )] -4   \ > \   R^3(1- 6 R^{-(1+i)})  \,
 \label{card2}
\end{eqnarray*}
 of the  $F_{n+1} $ rectangles  satisfy both conditions (H) and (I) for $n+1$.  Now choose exactly $ [R^3 \, ( 1
\; - \; 6 \, R^{-(1+i)} )]$ of these good rectangles and denote this collection by $ \cF^{\bm\theta} (F_n) $. Finally, define
 $$ \cF_{n+1}^{\bm\theta} \ := \ \bigcup_{F_n \subset \cF_n^{\bm\theta} } {\cal
F^{\bm\theta}}(F_n) \ \qquad {\rm and}  \qquad   \K \ := \
\bigcap_{n=0}^{\infty} \cF_n^{\bm\theta}     \, . $$

\medskip

\subsubsection{The finale}
It remains to show that
$$
 \dim \K  \ \to \  2  \qquad {\rm as } \qquad c \to 0 \, .
 $$
 This involves essentially following line by line the arguments set out in  \cite[\S3.2 and \S3.3]{PV}. The details are left to the reader.  \\[-1ex] \hspace*{\fill}$\boxtimes$

\vspace{7ex}

\noindent{\bf Acknowledgements.}  SV would like to thank Dani for his support during the early years of his mathematically life  -- basically when it matters most!  Naturally, an enormous thanks to the delightful and magnificent  Bridget, Ayesha and Iona.  The younger two are curious and challenging for all the right reasons -- long may you remain that way and good luck with high school!  This brings me onto their wonderful teachers during  the last two years at primary school who have installed them with great belief and have been superb role models for both children and adults. For this and much much more,   I would like to take this opportunity to thank  Mr Youdan, Mr Middleton and Mrs Davison.    I hope you continue to inspire!

\vspace{5mm}

\noindent Victor Beresnevich, \ Sanju Velani:\\ Department of Mathematics, University of
York,\\
Heslington, York, YO10 5DD, England.


\noindent{\em e-mails}:\\  {\tt victor.beresnevich@york.ac.uk}\\
{\tt sanju.velani@york.ac.uk}

\end{document}